\documentclass[12pt]{article}

\usepackage[utf8]{inputenc}				
\usepackage{amsmath}					
\usepackage{mathtools}					
\usepackage{amsthm,amssymb,wasysym}		
\usepackage{amsfonts}					
\usepackage{comment}					
\usepackage{authblk}					
\usepackage[babel]{csquotes}			
\usepackage{colonequals}				
\usepackage{url}						
\usepackage{graphicx}					
\usepackage{bbm}						
\usepackage{nicefrac}					
\usepackage[ngerman, english]{babel}
\makeatletter
\let\@fnsymbol\@arabic
\makeatother
\usepackage{geometry}
\usepackage[numbers,square]{natbib}		
\newcommand{\id}{{\boldsymbol{\mathbbm{1}}}}
\newcommand{\eps}{\varepsilon}

\newcommand{\R}{\mathbb{R}}
\newcommand{\N}{\mathbb{N}}

\newcommand{\C}{\mathbb{C}}
\DeclareMathOperator{\GL}{GL}

\DeclareMathOperator{\SO}{SO}

\DeclareMathOperator{\Sym}{Sym}
\DeclareMathOperator{\PSym}{PSym}
\newcommand{\GLn}{\GL(n)}

\newcommand{\SOn}{\SO(n)}

\DeclareMathOperator{\Log}{Log}

\DeclareMathOperator{\sym}{sym}

\DeclareMathOperator{\dev}{dev}
\DeclareMathOperator{\tr}{tr}

\DeclareMathOperator{\dist}{dist}

\DeclareMathOperator{\real}{Re}
\DeclareMathOperator{\imag}{Im}
\newcommand{\intd}[1]{{\mathrm{d}#1}}

\newcommand{\dd}[1]{\frac{\mathrm{d}}{\mathrm{d}#1}}

\newcommand{\norm}[1]{\Vert #1 \Vert}

\newcommand{\innerproduct}[1]{\langle #1 \rangle}

\newcommand{\tel}[1]{\frac{1}{#1}}
\newcommand{\half}{\tel{2}}
\newcommand{\WH}{W_{\textrm{H}}}
\usepackage[usenames,dvipsnames]{xcolor}
\usepackage{tikz}
\usetikzlibrary{arrows}
\usetikzlibrary{decorations.markings}
\usepackage{pgfplots}
\usetikzlibrary{calc}

\theoremstyle{plain}
\newcounter{theoremCounter}
\numberwithin{theoremCounter}{section}
\newtheorem{lemma}[theoremCounter]{Lemma}
\newtheorem{proposition}[theoremCounter]{Proposition}
\newtheorem{theorem}[theoremCounter]{Theorem}
\newtheorem{corollary}[theoremCounter]{Corollary}

\newtheorem{conjecture}[theoremCounter]{Conjecture}
\theoremstyle{definition}
\newtheorem{definition}[theoremCounter]{Definition}

\newtheorem{remark}[theoremCounter]{Remark}

\usepackage{wrapfig}

\theoremstyle{plain}
\newtheorem*{theorem*}{Theorem 1.2}

\newcommand{\distlogeuc}{\dist_{\text{\rm{log-Euclid}}}}
\newcommand{\dgPsym}{\dist_{\mathrm{geod,}\PSym(n)}}
\renewcommand{\PSym}{\mathop{\mathrm{Sym_+}}}

\begin{document}
\title{\vspace{-2.9em}The sum of squared logarithms inequality in arbitrary dimensions}
\author{Lev Borisov\footnote{Lev Borisov, Department of Mathematics, Rutgers University, 240 Hill Center, Newark, NJ 07102, United States, email: borisov@math.rutgers.edu}
	\quad and\quad Patrizio Neff\footnote{Patrizio Neff, Head of Lehrstuhl für Nichtlineare Analysis und Modellierung, Fakultät für Mathematik, Universität Duisburg-Essen, Thea-Leymann Str. 9, 45127 Essen, Germany, email: patrizio.neff@uni-due.de}
	\quad and\quad Suvrit Sra\footnote{Suvrit Sra, Laboratory for Information and Decision Systems, Massachusetts Institute of Technology, 77 Massachusetts Ave, Cambridge, MA 02139, United States, email: suvrit@mit.edu}
	\quad and\quad Christian Thiel\footnote{Corresponding author: Christian Thiel, Lehrstuhl für Nichtlineare Analysis und Modellierung, Fakultät für Mathematik, Universität Duisburg-Essen, Thea-Leymann Str. 9, 45127 Essen, Germany, email: christian.thiel@uni-due.de}}
\date{\vspace{-0.3em}\small\today}
\maketitle
\vspace{-.7cm}
\begin{abstract}
\noindent We prove the \emph{sum of squared logarithms inequality} (SSLI) which states that for nonnegative vectors $x, y \in \R^n$ whose elementary symmetric polynomials satisfy $e_k(x)\leq e_k(y)$ (for $1\leq k < n$) and $e_n(x)=e_n(y)$, the inequality $\sum_i (\log x_i)^2 \leq \sum_i(\log y_i)^2$ holds. Our proof of this inequality follows by a suitable extension to the complex plane. In particular, we show that the function $f\colon M\subseteq \C^n\to \R$ with $f(z)=\sum_i(\log z_i)^2$ has nonnegative partial derivatives with respect to the elementary symmetric polynomials of $z$. This property leads to our proof. We conclude by providing applications and wider connections of the SSLI.
\end{abstract}

\bigskip


\bigskip


\begingroup
\parskip=0.1em
\tableofcontents
\endgroup
%
%
\section{Introduction}
The \emph{sum of squared logarithms inequality} (SSLI) arose first as a scientific issue in 2012 \cite{neff2013henckyPAMM} while proving the following optimality result
\begin{equation}\label{question}
\inf_{Q\in\SO(n)} \| \sym \Log Q^TF\|^2\ = \inf_{Q\in\SO(n)}\ \ \inf_{\underset{\mathclap{\exp(Y)=Q^TF}}{Y\in\R^{n\times n}}}\| \sym Y\|^2\ =\ \| \log \sqrt{F^TF}\|^2\,,
\end{equation}
where $Y=\Log X$ denotes all solutions of the matrix exponential equation $\exp(Y) = X$, $\norm{\cdot}$ denotes the Frobenius matrix norm, and $\sym X \colonequals \half(X+X^T)$. 

The SSLI (formally stated in Theorem~\ref{theossli}) has been investigated in a series of works. In 2013, it was examined more closely by B\^{\i}rsan, Neff and Lankeit in \citep{Neff_log_inequality13}, who found a proof for $n\in\{2,3\}$. For $n=3$, the inequality can be written as follows: let $x_1,x_2,x_3,y_1,y_2,y_3 >0$ be positive real numbers such that
\begin{align*}
x_1+x_2+x_3\ &\leq\ y_1+y_2+y_3\,,\\
x_1\,x_2+x_1\,x_3+x_2\,x_3\ &\leq\ y_1\,y_2+y_1\,y_2+y_2\,y_3\,,\\
x_1\,x_2\,x_3\ &=\ y_1\,y_2\,y_3\,.
\end{align*}
Then, the sum of their squared logarithms satisfy the following inequality:
\begin{equation*}
(\log x_1)^2 + (\log x_2)^2 + (\log x_2)^2\ \leq\ (\log y_1)^2 + (\log y_2)^2 + (\log y_3)^2\,.
\end{equation*}
In 2015, \citet{PompeNeff2015} proved the SSLI for $n=4$, based on a new idea that did not extend to higher dimensions without further complications. To state the SSLI for arbitrary $n$, we first recall
\begin{definition}\label{defiElementarySymmetricPolynomial}
Let $x\in\R^n$. We denote by $e_k(x)$ the $k$-th \emph{elementary symmetric polynomial}, i.e.\ the sum of all $\binom nk$ products of exactly $k$ components of $x$:
\begin{equation}
e_k(x) \colonequals \sum_{1\le i_1<\ldots<i_k \le n}x_{i_1}x_{i_2}\ldots x_{i_k}\qquad\textrm{for any\ \,$k\in\{1,\ldots,n\}$}\,.\notag
\end{equation}
Note that $e_1(x)=x_1+x_2+\cdots+x_n$ and $e_n(x)=x_1\cdot x_2\cdot \ldots \cdot x_n$.
\end{definition}

We also write $\R_+ \colonequals \{x\in\R\,|\, x>0\}$ and $\R_- \colonequals \{x\in\R\,|\, x<0\}$ and set $\R_+^n = (\R_+)^n$.

\begin{theorem}[Sum of squared logarithms inequality]\label{theossli}
Let $n\in\N$ and $x, y \in \R_+^n$ such that 
\begin{align*}
e_k(x)\ &\leq\ e_k(y)\qquad \mathrlap{\textrm{for all\ \,$k\in\{1,\ldots,n-1\}$}},\\
\intertext{and}
~\\[-3.5em]
e_n(x)\ &\boldsymbol{=}\ e_n(y)\,.\\[1.05em]
\intertext{Then}
~\\[-4.55em]
\sum_{i=1}^n(\log x_i)^2\ &\leq\ \sum_{i=1}^n(\log y_i)^2\notag\,.
\end{align*}

This statement can equivalently be expressed as a minimization problem:

For $x\in\R_+^n$, let
\begin{equation}
\mathcal{E}_x\ \colonequals\ \bigl\{\,y\in\R_+^n\ |\ e_k(x)\leq e_k(y)\quad \textrm{for all\ \,$k\in\{1,\ldots,n-1\}$\ \ and\ \, $e_n(x)=e_n(y)$}\bigr\}\,.\notag
\end{equation}
Then
\begin{equation}
\inf_{y\in\mathcal{E}_x}\biggr\{ \sum_{i=1}^n(\log y_i)^2\biggr\}\ =\ \sum_{i=1}^n(\log x_i)^2\,.\notag
\end{equation}
\end{theorem}

Since $(\log y_i)^2\geq 0$ for all\ \,$i\in\{1,\ldots,n\}$, the expression is bounded below by $0$, so the infimum clearly exists. Note that $\mathcal{E}_x$ is a non-convex set.

\begin{remark}\label{counterexample}
If the equality assumption $e_n(x)=e_n(y)$ in the last elementary symmetric polynomial is replaced by the weaker requirement $e_n(x)\leq e_n(y)$, then the conclusion no longer holds in general. As a counterexample, consider $x=(e^{-1},\ldots,e^{-1})\in\R^n$ and $y=(1,\ldots,1)\in\R^n$; then $e_k(x) = \binom nk e^{-k} \leq \binom nk = e_k(y)$ for all $k\in\{1,\dotsc,n\}$, but $\sum_{i=1}^n(\log x_i)^2 = n > 0 = \sum_{i=1}^n(\log y_i)^2$.
\end{remark}

Neff, Nakatsukasa and Fischle \cite{Neff_Nagatsukasa_logpolar13} showed that the SSLI implies \eqref{question}.

The proof of the SSLI presented in this work was motivated by the second named author, who published the SSLI conjecture (at that point) on the internet platform \emph{MathOverflow} \cite{sslimathoverflow}. The first named author extended the problem to the complex plane and presented a sketch of a proof.


Miroslav {\v{S}}ilhav{\'y} (Czech Academy of Science) considered the problem after private communication with P. Neff and provided a characterization of functions that satisfy E-monotonicity.  Interestingly, shortly after seeing L. Borisov's solution, one of the authors (S. Sra) suggested via email that ``a full generalization of this idea should be possible via Pick-Nevalinna theory.'' This idea is natural, and the details were independently discovered and worked out by M. {\v{S}}ilhav{\'y} \cite{silhavy_2015}; it is also worth noting that actually \citet{jozsa2015symmetric} foreshadowed the Pick function based approach to proving such inequalities but did not develop it fully. Our remarks here merely outline the historical sequence of events (to our knowledge), and to highlight the remarkable fact that like many other problems in mathematics, the SSLI also witnessed several essentially simultaneous solutions; each exposing different aspects of it and thus contributing to our understanding.

In this paper we give a self-contained exposition of our new methods towards proving the SSLI.
%
%
\section{Proof of Theorem \ref{theossli}}
\label{section:proof}
In our further calculations, we will use the following lemma and the resulting corollary.

\begin{lemma}
Let $z_1,\ldots,z_n$ be \emph{pairwise different} complex numbers and $k\in \{0,\ldots,n-1\}$. Then
\begin{equation}\label{pfd}
\sum_{i=1}^n\frac{z_i^k}{(t-z_i)\prod_{\underset{j\neq i}{j=1}}^n(z_i-z_j)}\ =\ \frac{t^k}{\prod_{j=1}^n(t-z_j)}\qquad\textrm{for all $t\in\C\setminus\{z_1,\ldots,z_n\}$}\,.
\end{equation}
\end{lemma}
For $n=3$ and $k=2$, for example, the equality reads
\begin{equation*}\textstyle
\frac{a^2}{(a-b)(a-c)(t-a)} + \frac{b^2}{(b-a)(b-c)(t-b)} + \frac{c^2}{(c-a)(c-b)(t-c)} = \frac{t^2}{(t-a)(t-b)(t-c)}\,.
\end{equation*}
\begin{proof}
Let $k\in\{0,\ldots,n-1\}$. Then according to the theorem of partial fraction decomposition, there exist complex numbers $a_1,\ldots,a_n$ such that
\begin{equation*}
\frac{t^k}{\prod_{j=1}^n(t-z_j)}\ =\ \sum_{i=1}^n\frac{a_i}{t-z_i}\qquad\textrm{for all $t\in\C\setminus\{z_1,\ldots,z_n\}$}\,.
\end{equation*}
For given $r\in\{1,\ldots,n\}$, multiplying both sides of the equation with $t-z_r$ yields
\begin{equation}
\frac{t^k}{\prod_{\underset{j\neq r}{j=1}}^n(t-z_j)}\ =\ \sum_{i=1}^n\frac{t-z_r}{t-z_i}\,a_i\qquad\textrm{for all $t\in\C\setminus\{z_1,\ldots,z_n\}$}\,.
\end{equation}
Taking the limit $t\to z_r$ on both sides of the equality, we find
\begin{equation*}
\raisebox{0.4em}{$\displaystyle\frac{z_r^k}{\prod_{\underset{j\neq r}{j=1}}^n(z_r-z_j)}\ =\ a_r$}\,.\qedhere
\end{equation*}
\end{proof}

\begin{corollary}
Let $z_1,\ldots,z_n \in \C$ be \emph{pairwise different} complex numbers. Then
\begin{equation}
\sum_{i=1}^n\frac{z_i^k}{\prod_{\underset{j\neq i}{j=1}}^n(z_i-z_j)}\ =\ 0\qquad\textrm{for all $k\in \{0,\ldots,n-2\}$}\,.\label{pfd2}
\end{equation}
\end{corollary}
For example, we find for $n=4$ and $k=2$: 
\begin{equation*}\textstyle
\frac{a^2}{(a-b)(a-c)(a-d)} + \frac{b^2}{(b-a)(b-c)(b-d)} + \frac{c^2}{(c-a)(c-b)(c-d)} + \frac{d^2}{(d-a)(d-b)(d-c)}\ =\ 0\,.
\end{equation*}
\begin{proof}
Let $k\in\{0,\ldots,n-2\}$. Using equality \eqref{pfd}, we obtain
\begin{equation}
\sum_{i=1}^{n-1}\frac{z_i^k}{(t-z_i)\prod_{\underset{j\neq i}{j=1}}^{n-1}(z_i-z_j)}\ =\ \frac{t^k}{\prod_{j=1}^{n-1}(t-z_j)}\qquad\textrm{for all $t\in\C\setminus\{z_1,\ldots,z_{n-1}\}$}\,.
\end{equation}
Setting $t\colonequals z_n$ and rearranging the equation yields the statement.
\end{proof}

To introduce the basic idea of our proof, we first recall the relationship between a vector $z\colonequals (z_1,z_2,\ldots,z_n)$ and the vector of the elementary symmetric polynomials evaluated at $z$, i.e.\ $(e_1(z),\ldots,e_n(z))$. To this end, we define the characteristic polynomial $h_e$ of a linear map with the invariants $e_1,\ldots,e_n$:
\begin{equation*}
h_e(t)\ \colonequals\ t^n -e_1\,t^{n-1} + e_2\,t^{n-2}+\ldots+(-1)^n\,e_n\ =\ t^n + \sum_{k=1}^n(-1)^k\,e_k\,t^{n-k}\,.
\end{equation*}
Since if $h_e$ has the roots $z_1,\ldots,z_n$, we can write
\begin{align}
h_e(t)\ &=\ (t-z_1)(t-z_2)\ldots(t-z_n)\notag\\
&=\ t^n-(z_1+\ldots + z_n)t^{n-1} + (z_1\,z_2 + \ldots + z_{n-1}\,z_n)t^{n-2}+\ldots+(-1)^nz_1\ldots z_n\\
&=\ t^n-e_1(z)t^{n-1}+e_2(z)t^{n-2}+\ldots+(-1)^ne_n(z)\notag\,.
\end{align}

In this paper we will study different restrictions in the co-domain of the elementary symmetric polynomials. However, we will always assume that this co-domain is positive and real.
It is also convenient to introduce an ordering of the complex numbers in order to ensure the uniqueness of the coefficient vector corresponding to a given set of roots. We therefore define the set
\begin{equation*}
\C^{n\uparrow}\ \colonequals\ \bigl\{z\in\C^n\,|\ \real(z_1)\geq\ldots\geq\real(z_n)\,,\ \real z_i = \real z_{i+1} \Rightarrow \imag z_i \geq \imag z_{i+1}\ \forall i\in\{1,\ldots,n-1\}\bigr\}\,,
\end{equation*}
which contains only ordered vectors and thereby excludes all rearrangements. Furthermore, we define the set
\begin{equation*}
M\ \colonequals\ \bigl\{ z\in \C^{n\uparrow}\,|\ e_1(z),e_2(z),\ldots,e_n(z)\in\R_+\bigr\}
\end{equation*}
of all ordered vectors with exclusively positive elementary symmetric polynomials. In contrast to previous work on the SSLI, we extend our view directly to complex roots in $M$, which provides the crucial advantage.

\begin{lemma}\label{defphi}
The function $M\to\R_+^n$ that maps each vector $z\in M$ onto the coefficient vector $e$ corresponding to the uniquely determined polynomial $h_e$ with roots $z_1,\dotsc,z_n$ is continuous and bijective. Its inverse function is continuous as well, and we denote it by
\begin{equation*}
\varphi\colon \R_+^n\to M\subseteq \C^{n\uparrow}\,,\quad(e_1,\ldots e_n) \mapsto \varphi(e_1,\ldots, e_n)\,.
\end{equation*}

Furthermore, each vector $(z_1,\ldots,z_n) \in M$ contains only positive real numbers and complex conjugate pairs of numbers.
\end{lemma}
\begin{proof}

The elementary symmetric polynomials $e_1(z),\ldots,e_n(z)$ evaluated at $z$ are exactly the coefficients $e_1,\ldots,e_n$ of the polynomial $h_e$ with the roots $z_1,\ldots,z_n$. The elementary symmetric polynomials are obviously continuous.

On the other hand, applying the fundamental theorem of algebra, we know that $h_e$ has exactly $n$ complex roots, all of which are either real or complex conjugate pairs. It is easy to see that all real roots must be positive: since the polynomial $h_e(-t)\ =\ t^n + \sum_{k=1}^ne_k\,t^{n-k} >0$ for all $x\in\R_+$, because all $e_k$ are positive. Thus $h_e(-t)$ has no positive and therefore $h_e(t)$ has no negative real roots. A proof of the continuity of $\varphi$ is shown in \cite{cucker1989alternate}.
\end{proof}

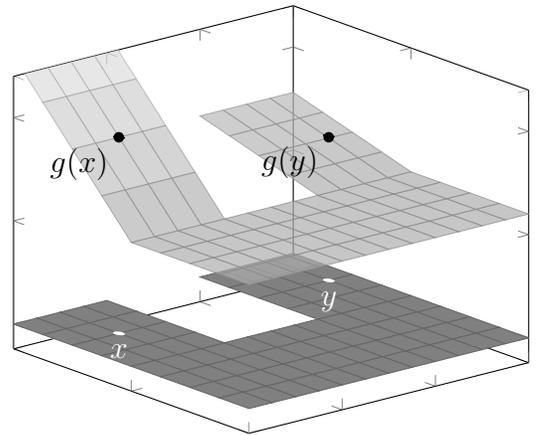
\begin{wrapfigure}[12]{r}{0.45\textwidth}
	\begin{center}\vspace{-2.0em}
    \begin{tikzpicture}
        \begin{scope}[scale=1]

\begin{scope}
\begin{axis}[samples=5, view={-40}{22},zmax=12,colormap={bw}{gray(0cm)=(0.5); gray(1cm)=(0.9)}, xticklabels = , yticklabels = , zticklabels = ]
\addplot3[surf, domain=0:1, y domain=0:1] {0};
\addplot3[surf, domain=1:2, y domain=0:1] {0};
\addplot3[surf, domain=2:3, y domain=0:1] {0};
\addplot3[surf, domain=0:1, y domain=1:2] {0};
\addplot3[surf, domain=2:3, y domain=1:2] {0};
\fill[white] (50,150,0) circle (4) node[below] {$x$};
\fill[white] (275,150,0) circle (4) node[below] {$y$};
\addplot3[surf, domain=0:1, y domain=0:1, opacity=0.7] {0+6};
\addplot3[surf, domain=1:2, y domain=0:1, opacity=0.7] {0+6};
\addplot3[surf, domain=2:3, y domain=0:1, opacity=0.7] {0+6};
\addplot3[surf, domain=0:1, y domain=1:2, opacity=0.7] {7*y-7+6};
\addplot3[surf, domain=2:3, y domain=1:2, opacity=0.7] {1.8*y-2+6.2};
\fill (50,150,95) circle (4) node[below left] {$g(x)$};
\fill (275,150,70) circle (4) node[below left] {$g(y)$};
\end{axis}
\end{scope}
\fill (1.4,3.94) circle (.07);
\fill (4.19,3.94) circle (.07);

\end{scope}
    \end{tikzpicture}
    \vspace*{-2em}
	\end{center}
	\caption{\label{fig:counterExampleGraph}The graph of a function $g$ with non-negative partial derivatives but non-convex domain of definition; note that $g(x)>g(y)$, although $x_i\leq y_i$ for $i\in\{1,2\}$.}
\end{wrapfigure}

We now come to the proof of the SSLI. The main idea has already been pursued in prior attempts to prove the inequality: instead of directly working with the function $f(z)\colonequals\sum_{i=1}^n(\log z_i)^2$ on the set $M$ of roots, we consider the composition $f\circ \varphi$ which depends on the elements $e\in T$ of a suitable set of coefficients $T\subseteq \R_+^n$.
Of course, we have to choose $T$ in a way such that $(f\circ \varphi)(e)\in\R$ for all $e\in T$.

The proof of the SSLI now can be divided into two steps:
\begin{itemize}
\item[1.)] We show that $\displaystyle \frac{\partial (f\circ\varphi)}{\partial e_k} \geq 0$\,.
\item[2.)] We find a path $\gamma\colon [0,1]\to \varphi(T)$ with $\gamma(0)=x$, $\gamma(1)=y$ such that $\displaystyle\dd{s}e_k\bigl(\gamma(s)\bigr)\geq 0$ for all $s\in(0,1)$ and $k\in\{1,\ldots,n-1\}$ as well as $\displaystyle\dd{s}e_n\bigl(\gamma(s)\bigr)=0$ for all $s\in(0,1)$\,.
\end{itemize}

Note carefully that condition 1.) alone is not sufficient. To understand this, let us consider the graph of a function $g\colon D\to\R$ with a non-convex domain $D\subseteq\R^n$ as shown in Figure \ref{fig:counterExampleGraph}.

As we see, even though the function only has non-negative partial derivatives, it is \emph{not} true that $g(x)\leq g(y)$ for each pair $x,y\in D$ such that (componentwise) $x\leq y$. In particular, we cannot find a path connecting $x$ to $y$ which is increasing in all components.

We therefore want to choose an appropriate domain $T\subseteq \R_+^n$ in order to prevent these complications. The problem is that for a too restricted choice of $T$, we do not easily find suitable paths satisfying condition 2.), while if $T$ is chosen too large, it becomes more difficult to prove that condition 1.) holds on all of $T$.

In \cite{PompeNeff2015}, Pompe and Neff managed to prove the SSLI for $n=4$ by choosing
\begin{equation}
T\ =\ \bigl\{\,\bigl( e_1(z),\ldots,e_n(z)\bigr)\ |\ z\in\R_+^n\, \bigr\}\,;
\end{equation}
the authors in fact do not to limit the inequality to the function $f(z)\colonequals\sum_{i=1}^n(\log z_i)^2$, but to prove it for a whole class of functions. For these $f$, they show condition 1.), i.e.\ the non-negativity of the partial derivatives with respect to the $k$-th elementary symmetric polynomial, by showing for each point on the path from condition 2.) that
\begin{equation}
\widehat D\ \colonequals\ \sum_{j=1}^nf'(z_j)(-z_j)^{n-k}\biggl(\prod_{\underset{j\neq i}{j=1}}^n(z_j-z_i)\biggr)^{-1}\ \geq\ 0\,.
\end{equation}
Although their choice of the domain $T$ is not convex, they demonstrate the existence of paths from $x$ to $y$ without constructing them explicitly: they show that a path which satisfies condition 2.) can be continued and must therefore reach $y$ after starting at $x$. However, for $n>4$, this method seems impractical due to the amount of computational effort required. 

In this paper, we consider the convex set 
\begin{equation}
T\colonequals\R_+^n\,.
\end{equation}
By this choice, satisfying condition 2.) is rather trivial; we take the straight line $\tilde \gamma$ in $T$ from $e(x)$ to $e(y)$, i.e.\ $\tilde\gamma(s)\colonequals s\cdot e(x) + (1-s)\cdot e(y)$, and consider the curve $\gamma:[0,1]\to M$ with $\gamma(s)=\varphi(\tilde\gamma(s))$. Of course, it is difficult to explicitly characterize $\gamma$, but this is not necessary for our proof. It therefore only remains to verify condition 1.), which requires some more elaborate methods.

{\abovedisplayskip=4pt\belowdisplayskip=4pt
As indicated earlier, we define the function
\begin{equation}\label{eq:deff}
f\colon (\C\setminus(\R_- \cup \{0\})^n \to \C\qquad\textrm{with $f(z_1,\ldots,z_n) \colonequals \sum_{i=1}^n(\log z_i)^2$}\,,
\end{equation}}
where $\log$ denotes the main branch of the complex logarithm, which is defined for $\C\setminus (-\infty,0]$.

\begin{lemma}
The composition $f\circ \varphi$ is a real-valued function on $\R_+^n$, and takes the form
\begin{equation*}
f\circ \varphi\colon \R_+^n\to \R\qquad\textrm{with $(f\circ \varphi)(e) \colonequals f\bigl(\varphi_1(e),\ldots,\varphi_n(e)\bigr)$}\,.
\end{equation*}
\end{lemma}
\begin{proof}

From Lemma \ref{defphi} we already know that the components of $\bigl(\varphi_1(e),\ldots,\varphi_n(e)\bigr)$ are either real or pairs of complex conjugates. For such pairs,
\begin{equation*}
(\log \varphi_i(e))^2 + (\log \overline{\varphi_i(e)})^2\ =\ (\log \varphi_i(e))^2 + (\overline{\log \varphi_i(e)})^2\ =\ (\log \varphi_i(e))^2 + \overline{(\log \varphi_i(e))^2}\ \in \R\,.
\end{equation*}
Since $(\log \varphi_i(e))^2$ is real-valued for all real components $\varphi_i(e)$ as well, it follows that
\begin{equation*}
(f\circ \varphi)\bigl(e_1(z),\ldots,e_n(z)\bigr) = \sum_{i=1}^n(\log z_i)^2 \in\R\,.\qedhere
\end{equation*}
\end{proof}

In order to prove condition 1.), we need to compute the inner derivative of the composition $f\circ \varphi$:
\begin{proposition}\label{diffableprop}
Let $e\colonequals(e_1,\ldots,e_n)\in\R_+^n$ be such that all components of $\varphi(e)$ (i.e.\ the roots of $h_e$) are pairwise different. Then $\varphi$ is differentiable at $e$ with partial derivative
\begin{equation*}
\frac{\partial \varphi_i}{\partial e_k}(e)\ =\ \frac{(-1)^{k+1}\,\varphi_i(e)^{n-k}}{\prod_{\underset{j\neq i}{j=1}}^n\bigl(\varphi_i(e) - \varphi_j(e)\bigr)}\qquad\textrm{for any $k\in\{1,\ldots,n\}$}\,.
\end{equation*}
\end{proposition}
\begin{proof}
Using the notation $h(e_1,\ldots,e_n,t)\colonequals h_e(t)$, we can characterize the functions $\varphi_i$, $i\in\{1,\ldots,n\}$, implicitly by the equation
\begin{equation*}
0\ =\ h\bigl(e_1,\ldots,e_n,\varphi_i(e_1,\ldots,e_n)\bigr)\qquad\textrm{for all $e=(e_1,\ldots,e_n)\in\R_+^n$}\,.
\end{equation*}
According to the implicit function theorem, $\varphi_i$ is differentiable at a point $e\in\R_+^n$ if $\frac{\partial h}{\partial t}\bigl(e,\varphi_i(e)\bigr) \neq 0$. By assumption, the roots $\varphi_1(e),\ldots,\varphi_n(e)$ of $h_e$ are pairwise different; in particular, the root $\varphi_i(e)$ is simple and therefore $\frac{\partial}{\partial t}h(e,\varphi_i(e))\neq 0$. We differentiate
\begin{align}
0\ =\ \dd{e_k}h\bigl(e,\varphi_i(e)\bigr)\ &=\ \sum_{j=0}^{n-1}\left(\frac{\partial h}{\partial e_j}\right)(e,\varphi_i(e))\cdot \delta_{j,k} + \frac{\partial h}{\partial t}\bigl(e,\varphi_i(e)\bigr)\cdot\frac{\partial \varphi_i}{\partial e_k}(e)\\
&=\ (-1)^k\varphi_i(e)^{n-k}\cdot 1 + \frac{\partial h}{\partial t}\bigl(e,\varphi_i(e)\bigr)\cdot\frac{\partial \varphi_i}{\partial e_k}(e)\,,\notag
\end{align}
and rearrangement yields the desired equality, since $\frac{\partial h}{\partial t}\bigl(e,\varphi_i(e)\bigr) = \prod_{\underset{j\neq i}{j=1}}^n\bigl(\varphi_i(e) - \varphi_j(e)\bigr)$.
\end{proof}

After this preliminary work, we can formulate condition 1.) in our context:

\begin{conjecture}\label{parpos}
Let $e=(e_1,\ldots,e_n)\in\R_+^n$. Then the composition $f\circ \varphi$ is differentiable and the partial derivatives $\frac{\partial(f\circ\varphi)}{\partial e_k}(e)$ for $k\in\{1,\ldots,n-1\}$ are positive, i.e.
\begin{equation*}
\frac{\partial(f\circ\varphi)}{\partial e_k}(e)\ >\ 0\qquad\textrm{for all $e\in\R_+^n$ and any $k\in\{1,\ldots,n-1\}$}\,.
\end{equation*}
\end{conjecture}

We will prove Conjecture \ref{parpos} under the additional assumption that all components of $\varphi(e)$ are \emph{pairwise different} (i.e.\ that $h_e$ has only simple roots). With this restriction, the assertion does not imply condition 1.), but we will compensate this drawback with some additional work in Proposition \ref{endlnst} as well as Lemmas \ref{sslipaarwuntersch} and \ref{sslistrikt}.
\begin{lemma}\label{lem1}
Let $e=(e_1,\ldots,e_n)\in\R_+^n$ be such that all components of $\varphi(e)$ are \emph{pairwise different}. Then
\begin{equation}\label{partablsumme}
\frac{\partial(f\circ \varphi)}{\partial e_k}(e)\ =\ -2\,\sum_{i=1}^n\frac{\bigl(-\varphi_i(e)\bigr)^{n-k-1}}{\prod_{\underset{j\neq i}{j=1}}^n\bigl(\varphi_j(e)-\varphi_i(e)\bigr)}\log \varphi(e)\,.
\end{equation}
\end{lemma}
\begin{proof}
Because the components of $\varphi(e)$ are pairwise different, $\varphi$ is differentiable at $e$ according to Proposition \ref{diffableprop}. Then $f\circ \varphi$ is differentiable at $e$ as well, and we can determine the partial derivatives for $k\in\{1,\ldots,n\}$ using the chain rule:
\begin{align}
\frac{\partial(f\circ \varphi)}{\partial e_k}(e)\ &=\ \sum_{i=1}^n\frac{\partial f}{\partial z_i}\bigl(\varphi_1(e),\ldots,\varphi_n(e)\bigr)\,\frac{\partial \varphi_i}{\partial e_k}(e_1,\ldots,e_n)\notag\\
&=\ \sum_{i=1}^n\frac{2\,\log \varphi_i(e)}{\varphi_i(e)}\,\frac{(-1)^{k+1}\,\varphi_i(e)^{n-k}}{\prod_{\underset{j\neq i}{j=1}}^n\bigl(\varphi_i(e)-\varphi_j(e)\bigr)}\notag\\
&=\ 2\,\sum_{i=1}^n\frac{(-1)^{n+k}\,\varphi_i(e)^{n-k-1}}{(-1)^{n-1}\prod_{\underset{j\neq i}{j=1}}^n\bigl(\varphi_i(e)-\varphi_j(e)\bigr)}\log \varphi_i(e)\\
&=\ -2\,\sum_{i=1}^n\frac{\bigl(-\varphi_i(e)\bigr)^{n-k-1}}{\prod_{\underset{j\neq i}{j=1}}^n\bigl(\varphi_j(e)-\varphi_i(e)\bigr)}\log \varphi_i(e)\,.\qedhere
\end{align}
\end{proof}

{\abovedisplayshortskip=2pt\belowdisplayshortskip=4pt
We now want to show that the partial derivatives of $f\circ\varphi$ are positive. The expression calculated in Lemma \ref{lem1} also appeared in other work. For example, Mitchison and Josza in 2004 point out in the appendix of \cite{mitchison2004towards} that
\begin{equation}
\label{eq:mitchJos}
(-1)^q\,\sum_{i=1}^n\frac{z_i^{n-q}\log z_i}{\prod_{\underset{j\neq i}{j=1}}^n(z_j-z_i)}\ \geq\ 0
\end{equation}
for all $q\in{2,\ldots,n}$. If we let $z_i = \varphi_i(e)$, we seem to have already reached our goal. However, Mitchison und Josza prove inequality \eqref{eq:mitchJos} only for the case that $z_1,\ldots,z_n$ are the eigenvalues of a Gram matrix, which is necessarily symmetric and positive definite. Their ineqality can therefore only be applied to positive real numbers $z_1,\ldots,z_n$, while in our case, $\varphi_1(e),\ldots,\varphi_n(e)$ might also include pairs of complex conjugates.}

We close this gap by showing:

\begin{lemma}\label{lem2}
Let $(z_1,\ldots,z_n)\in M$. Then for all $r\in\{0,\ldots,n-2\}$,
\begin{equation*}
-\sum_{i=1}^n\frac{(-z_i)^r}{\prod_{\underset{j\neq i}{j=1}}^n(z_j-z_i)}\,\log z_i\ =\ \int_0^\infty\frac{t^r}{\prod_{j=1}^n(t+z_j)}\intd t\ \boldsymbol{>}\ 0\,.
\end{equation*}
\end{lemma}
\begin{proof}
We observe that
\begin{equation*}
\prod_{j=1}^n(t+z_j) \ =\ t^n + \sum_{k=2}^{n-1} e_k(z)\,t^{n-k} + \prod_{j=1}^nz_j \ \geq\ \min\bigg\{t^n,\prod_{j=1}^nz_j\bigg\} \ >\ 0
\end{equation*}
for all $t\in\R_+$. Therefore, since $r\leq n-2$,
\begin{equation}
\int_0^\infty\left|\frac{t^r}{\prod_{j=1}^n(t+z_j)}\right|\intd t\ \leq\ \int_0^1\frac{t^r}{\prod_{j=1}^nz_j}\intd t + \int_1^\infty \underbrace{t^{r-n}}_{\leq t^{-2}}\intd t\ \leq\ \frac 1{\prod_{j=1}^nz_j}\, +\ 1\,.
\end{equation}

Thus the integral converges for $r\in\{0,\ldots,n-2\}$ and, since $\frac{t^r}{\prod_{j=1}^n(t+z_j)}> 0$, its value is positive.

We now use \eqref{pfd} and exchange the order of summation and integration, where in evaluating the integral at its ``upper boundary'' we use that $\log(t+z_i) =\log t+\log(1+\frac{z_i}t)$, whereas this decomposition is not applied at $t=0$:
\begin{align}
&\int_0^\infty \frac{t^r}{\prod_{j=1}^n(t+z_j)}\intd t\ \overset{\mathclap{\eqref{pfd}}}=\ \int_0^\infty\sum_{i=1}^n\frac{(-z_i)^r}{(t+z_i)\prod_{\underset{j\neq i}{j=1}}^n\bigl((-z_i)-(z_j)\bigr)}\intd t\notag\\
&=\ \sum_{i=1}^n\frac{(-z_i)^r}{\prod_{\underset{j\neq i}{j=1}}^n(z_j-z_i)}\ \cdot\ \underbrace{\log(t+z_i)}_{\mathclap{=\log t+\log(1+\frac{z_i}t)}}\ \Bigr|_{t=0}^{t\to\infty}\\
&=\ \lim_{t\to\infty}\Biggl(\sum_{i=1}^n\frac{(-z_i)^r}{\prod_{\underset{j\neq i}{j=1}}^n(z_j-z_i)}\log t + \sum_{i=1}^n\frac{(-z_i)^r}{\prod_{\underset{j\neq i}{j=1}}^n(z_j-z_i)}\left( \log\left(1+\frac{z_i}t\right)\right) - \sum_{i=1}^n\frac{(-z_i)^r}{\prod_{\underset{j\neq i}{j=1}}^n(z_j-z_i)}\log z_i\Biggr)\,.\notag
\end{align}
It is easy to see that $\lim_{t\to\infty}\log\left(1+\frac{z_i}t\right)=0$. 

This immediately implies that $\sum_{i=1}^n\frac{(-z_i)^r}{\prod_{\underset{j\neq i}{j=1}}^n(z_j-z_i)}=0$, because otherwise $\lim_{t\to\infty}\sum_{i=1}^n\frac{(-z_i)^r}{\prod_{\underset{j\neq i}{j=1}}^n(z_j-z_i)}\log t$ would diverge, in contradiction to the already established convergence of the integral. This completes the proof.
\end{proof}

\begin{proof}[Proof of Conjecture \ref{parpos} in the case of pairwise different roots $\varphi_i(e)$.]\ \\
We can now conclude: for all $n-k-1\in\{0,\ldots,n-2\}$, that is for $k\in\{1,\ldots,n-1\}$,
\begin{align*}
\frac{\partial(f\circ \varphi)}{\partial e_k}(e)\ &\overset{\mathrm{Prop. \ref{lem1}}}{=}\ -2\,\sum_{i=1}^n\frac{\bigl(-\varphi_i(e)\bigr)^{n-k-1}}{\prod_{\underset{j\neq i}{j=1}}^n\bigl(\varphi_j(e)-\varphi_i(e)\bigr)}\log \varphi(e)\\ &\overset{\mathrm{Prop. \ref{lem2}}}{=}\ 2\,\int_0^\infty \frac{t^{n-k-1}}{\prod_{j=1}^n(t+\varphi_j(e))}\intd t\ >\ 0\,.\qedhere
\end{align*}
\end{proof}

This shows that condition 1.) holds on nearly the entire domain $T$; the only problems occur for those points $e\in T$ for which $h_e$ has multiple roots. The set $T$ is indeed convex, meaning that we can easily construct a path as described in condition 2.), but since such a path may pass through points with multiple roots, it is not necessarily differentiable everywhere. However, as the next proposition shows, under suitable assumptions a straight line in $T$ contains at most finitely many of these problematic points:
\begin{proposition}\label{endlnst}
Let $p_0$ and $p_1$ be polynomials of degree $n$ such that at least one of them has only simple roots. Then there are only finitely many $s\in[0,1]$ such that the polynomial $p_s\colonequals (1-s)\,p_0+s\,p_1$ has multiple roots.
\end{proposition}
\begin{proof}
Let $p=\alpha\,\prod_{i=1}^n(X-x_i) = \sum_{i=0}^na_i\,X^{n-i}$ be a polynomial of degree $n$ with roots $z_1,\ldots,z_n$ and coefficients $a_0,\ldots,a_n$. The discriminant of $p$ is defined as
\begin{equation}
D(p)\ \colonequals\ \prod_{1\leq i<j\leq n}(z_i-z_j)^2
\end{equation}
and is zero if and only if $p$ has multiple roots (see also \cite[p. 204]{lang2002algebra}). The discriminant is a symmetric homogeneous polynomial in the variables $z_1,\ldots,z_n$ and thus can be expressed as a homogeneous polynomial in $e_k(z_1,\ldots,z_n)$, i.e.\ in the coefficients $a_k$.\footnote{For example, the polynomial $q(t)=t^3-a\,t^2+b\,t-c$ has the discriminant $D(q) = a^2\,b^2-4\,b^3-4\,a^3\,c+18\,a\,b\,c - 27\,c^2$.}

The coefficients of $p_s$ are polynomials in $s$, so $D(s)\colonequals D(p_s)$ is also a polynomial in $s$. If $D(s)$ is the zero polynomial, then both $p_0$ and $p_1$ must have multiple roots. On the other hand, if both $p_0$ and $p_1$ do not have multiple roots, then $D(s)$ is not the zero polynomial and is zero only for finitely many $s\in[0,1]$, and thus $p_s$ can have multiple roots only for finitely many $s\in[0,1]$.
\end{proof}

For vectors $x,y\in\R_+^n$, we can therefore directly prove the SSLI using conditions (1) and (2) only under the additional assumption that at least one of the two vectors has pairwise different components. In order to circumvent this limitation, we first show a strict version of the SSLI:

\begin{lemma}\label{sslipaarwuntersch}
Let $x=(x_1,\ldots,x_n),\ y=(y_1,\ldots,y_n)\in\R_+^n$ be such that $x_1>x_2>\ldots>x_n$, $y_1\geq y_2\geq\ldots\geq y_n$,\; $e_k(x)<e_k(y)$ for all $k\in\{1,\ldots,n-1\}$ and $e_n(x)=e_n(y)$. Then $f(x)<f(y)$ with $f$ as in \eqref{eq:deff}, i.e.
\begin{equation*}
f(x)\ =\ \sum_{i=1}^n(\log x_i)^2\ <\ \sum_{i=1}^n(\log y_i)^2\ =\ f(y)\,.
\end{equation*}
\end{lemma}
\begin{proof}
Consider the path $e^s = (e^s_1,\ldots,e^s_n) \subseteq \R_+^n$ for $s\in[0,1]$ with
\begin{equation*}
e_k^s\ \colonequals\ (1-s)\,e_k(x) + s\,e_k(y)\,.
\end{equation*}
Then $e^0=e(x)$ and $e^1=e(y)$ as well as $e_k(x) < e_k(y)$ for all $k\in\{1,\ldots,n-1\}$ and $e_n(x)=e_n(y)$. Since $x_1>x_2>\ldots>x_n$, the polynomial $h_{e^0}$ has pairwise different roots. According to Proposition \ref{endlnst}, there is only a finite number of $s\in[0,1]$ such that $h_{e^s}$ has pairwise different roots. Then by Conjecture \ref{parpos}, $f\circ \varphi$ is differentiable at all but finitely many points along the path $s\mapsto e_s$, and
\begin{equation}
\dd{s}\,f\bigl(\varphi(e_s)\bigr) \ =\ \sum_{k=1}^n\frac{\partial (f\circ \varphi)}{\partial e_k}(e^s)\cdot\frac{\mathrm{d}}{\mathrm{d}s}e_k^s\ =\ \sum_{k=1}^n \; \underbrace{\frac{\partial(f\circ \varphi)}{\partial e_k}(e^s)}_{>0} \cdot \underbrace{\vphantom{\frac{\partial(f\circ \varphi)}{\partial e_k}(e)}\bigl(e_k(y)-e_k(x)\bigr)}_{>0} \ >\ 0\,.
\end{equation}
Thus the continuity of the mapping $s\mapsto (f\circ \varphi)(e_s)$ implies its strict monotonicity on $[0,1]$, and therefore
\begin{equation*}
f(x) = (f\circ \varphi)\bigl(e(x)\bigr) = f\bigl(\varphi(e_0)\bigr) < f\bigl(\varphi(e_1)\bigr) = (f\circ \varphi)\bigl(e(y)\bigr) = f(y)\,.\qedhere
\end{equation*}
\end{proof}

In the next step, we use the continuity of the elementary symmetric polynomials in order to show the strict inequality for those $x,y$ whose components are not pairwise different.

\begin{lemma}\label{sslistrikt}
Let $x=(x_1,\ldots,x_n),\ y=(y_1,\ldots,y_n)\in\R_+^n$ be such that $e_k(x)<e_k(y)$ for all $k\in\{1,\ldots,n-1\}$ and $e_n(x)=e_n(y)$. Then $f(x)<f(y)$ with $f$ as in \eqref{eq:deff}, i.e.
\begin{equation*}
f(x)\ =\ \sum_{i=1}^n(\log x_i)^2\ <\ \sum_{i=1}^n(\log y_i)^2\ =\ f(y)\,.
\end{equation*}
\end{lemma}
\begin{proof}
If the components of $x=(x_1,\ldots,x_n)$ are pairwise different, then we can assume $x_1>x_2>\ldots>x_n$ and $y_1\geq y_2\geq\ldots\geq y_n$ after rearrangement. In this case, Lemma \ref{sslipaarwuntersch} shows that the inequality holds.

Otherwise, we need to slightly change the identical component pairs of $x$ in order to apply Lemma \ref{sslipaarwuntersch}. Because of the continuity of the elementary symmetric polynomials, if the changes to the components of $x$ are sufficiently small, then the resulting vector $x'$ still satisfies the inequality $e_k(x')<e_k(y)$ for all $k\in\{1,\ldots,n-1\}$. In order to preserve the equality $e_n(x')=e_n(y)$, we choose $x'_i\colonequals x_i\, (1+\eps)$ and $x'_j\colonequals x_j\,\frac 1{1+\eps}$ for each component pair $x_i=x_j$ and small $\eps>0$. So we find
\begin{align}
&(\log x'_i)^2 + (\log x'_j)^2\ =\ \bigl(\log x_i + \log(1+\eps)\bigr)^2 + \bigl(\log x_j - \log(1+\eps)\bigr)^2\notag\\
&\qquad=\ 2(\log x_i)^2 + 2\bigr(\log (1+\eps)\bigr)^2\ >\ (\log x_i)^2 + (\log x_j)^2\,.
\end{align}

If we dissolve all pairwise equalities in this way, then we can apply Lemma \ref{sslipaarwuntersch} to find
\begin{equation*}
\sum_{i=1}^n(\log x_i)^2\ <\ \sum_{i=1}^n(\log x'_i)^2\ \overset{\mathclap{\ref{sslipaarwuntersch}}}<\ \sum_{i=1}^n(\log y_i)^2\,.\qedhere
\end{equation*}
\end{proof}

Using the continuity of the elementary symmetric polynomials and the logarithmic function we are finally able to extend Lemma \ref{sslistrikt} and thus prove the SSLI without any restrictions:
\begin{theorem*}[Sum of squared logarithms inequality]
Let $n\in\N$ and $x_1,x_2,\ldots,x_n,y_1,y_2,\ldots,y_n\in\R_+$ such that\\[-1.5em]
\begin{align*}
e_k(x)\ &\leq\ e_k(y)\qquad \mathrlap{\textrm{for all\ \,$k\in\{1,\ldots,n-1\}$}}\\
\intertext{and}
~\\[-3.5em]
e_n(x)\ &\boldsymbol{=}\ e_n(y)\,.\\[1.05em]
\intertext{Then}
~\\[-4.55em]
\sum_{i=1}^n(\log x_i)^2\ &\leq\ \sum_{i=1}^n(\log y_i)^2\notag\,.
\end{align*}
\end{theorem*}
\begin{proof}
Choose $S\subseteq\{1,\ldots,n-1\}$ such that $e_k(x)=e_k(y)$ for all $k\in S$ and $e_k(x) < e_k(y)$ for all $k\in\{1,\ldots,n-1\}\setminus S$. Furthermore, for $k\in\{1,\dotsc,n\}$ and $m\in\N$, let
\begin{equation*}
e^m_k =\begin{cases}
e_k(x) - \frac 1m \quad~&\text{if }\; k\in S\,,\\
e_k(x) &\text{otherwise}
\end{cases}\qquad\quad\textrm{and}\qquad\quad x^m\colonequals\varphi(e^m)\,.
\end{equation*}
Then $0<e_k(x^m) < e_k(y)$ for all $k\in\{1,\ldots,n-1\}$ and all sufficiently large $m\in\N$ as well as $e_n(x^m)=e_n(y)$, thus Proposition \ref{sslistrikt} yields $\sum_{i=1}^n(\log x^m_i)^2\ <\ \sum_{i=1}^n(\log y_i)^2$ for all sufficiently large $m\in\N$. Since $\lim_{m\to\infty}x^m = x$, we find
\begin{equation*}
\sum_{i=1}^n(\log x_i)^2\ =\ \lim_{m\to\infty}\sum_{i=1}^n(\log x^m_i)^2\ \leq\ \sum_{i=1}^n(\log y_i)^2\,.\qedhere
\end{equation*}
\end{proof}

%
%
\section{Applications and Connections}
\label{sec:appl}

%
%
\subsection{Relation to entropy}
\citet{jozsa2015symmetric} study entropy and ``subentropy'' from the perspective of quantum information theory. They introduce and investigate partial derivatives of these quantities with respect to $e_k$, and establish their unrestricted nonnegativity. They view entropy as a function of the symmetric polynomials $e_k$ and use  analytic continuation to extend the definition to the entire set of nonnegative real $e_k$. Consequently, they obtain integral representations of entropy (and subentropy) from which the desired nonnegativity properties follow. More interestingly, they study higher partial derivatives w.r.t.\ the elementary symmetric polynomials and establish complete monotonicity of entropy and subentropy. Similar higher order monotonicity properties can be established for $f(x)=\sum_{i=1}^n (\log x_i)^2$.

In \cite{DannanNeffThiel2015} Dannan, Neff und Thiel discussed applications of the SSLI towards the entropy of probability distributions. We now prove a statement very similar to the entropy expression of two vectors:
we repeat our procedure from the last section with
\begin{equation}
g\colon (\C\setminus \R_{\leq 0})^n\to\C\,,\quad g(z_1,\ldots,z_n) = -\sum_{i=1}^nz_i\,\log (-z_i)
\end{equation}
instead of $f$, but otherwise identical notation and definitions. The composition $g\circ \varphi$ on $\R_+^n$ is once more a real-valued function: $x\in\R_-$ and $y\in\R_{\neq 0}$ imply $x\,\log (-x)\in\R$, while for complex conjugate pairs we find
\begin{equation*}
x_i\log x_i + \overline{x_i}\log \overline{x_i}\ =\ x_i\log x_i + \overline{x_i}\cdot\overline{\log x_i}\ =\ x_i\log x_i + \overline{x_i\log x_i}\ \in\R\,.
\end{equation*}
Analogously to the last section, the function $g \circ \varphi$ can be expressed as
\begin{equation*}
(g \circ \varphi)\colon \R_+^n\to \R\,,\quad (f\circ \varphi)(e) = f(\varphi_1(e),\ldots,\varphi_n(e)) =-\sum_{i=1}^n\varphi_i(e)\,\log \varphi_i(e)\,,
\end{equation*}
and for $x_1,\ldots,x_n\in\R_+$ we have $(g \circ \varphi)\bigl( e_1(x), e_2(x), \ldots, e_n(x)\bigr) = -\sum_{i=1}^nx_i\,\log x_i$.

We now determine the partial derivatives:
\begin{align}
\frac{\partial(g\circ \varphi)}{\partial e_k}(e)\ &=\quad \sum_{i=1}^n\frac{\partial g}{\partial z_i}\bigl(\varphi_1(e),\ldots,\varphi_n(e)\bigr)\,\frac{\partial \varphi_i}{\partial e_k}(e_1,\ldots,e_n)\notag\\
&\overset{\mathclap{\mathrm{Prop. \ref{diffableprop}}}}{=}\quad \sum_{i=1}^n\bigl(\log \varphi_i(e)+1\bigr)\,\frac{(-1)^{k+1}\,\varphi_i(e)^{n-k}}{\prod_{\underset{j\neq i}{j=1}}^n\bigl(\varphi_i(e)-\varphi_j(e)\bigr)}\notag\\
&=\quad \sum_{i=1}^n\frac{(-1)^{n+k-1}\,\varphi_i(e)^{n-k}}{(-1)^{n-1}\prod_{\underset{j\neq i}{j=1}}^n\bigl(\varphi_i(e)-\varphi_j(e)\bigr)}\bigl(\log \varphi_i(e)+1\bigr)\\
&=\quad -\,\sum_{i=1}^n\frac{\bigl(-\varphi_i(e)\bigr)^{n-k}}{\prod_{\underset{j\neq i}{j=1}}^n\bigl(\varphi_j(e)-\varphi_i(e)\bigr)}\log \varphi_i(e)\ -\,\sum_{i=1}^n\frac{\bigl(-\varphi_i(e)\bigr)^{n-k}}{\prod_{\underset{j\neq i}{j=1}}^n\bigl(\varphi_j(e)-\varphi_i(e)\bigr)}\,.\notag
\end{align}
By \eqref{pfd2}, the second sum is zero for $n-k\in\{0,\ldots,n-2\}$, i.e.\ for $k\in\{2,\ldots,n\}$. Thus we obtain
\begin{equation*}
\frac{\partial(g\circ \varphi)}{\partial e_k}(e)\ =\ -\,\sum_{i=1}^n\frac{\bigl(-\varphi_i(e)\bigr)^{n-k}}{\prod_{\underset{j\neq i}{j=1}}^n\bigl(\varphi_j(e)-\varphi_i(e)\bigr)}\log \varphi_i(e)\overset{\mathrm{Lemma \ref{lem2}}}{>}\ 0
\end{equation*}
for $n-k\in\{0,\ldots,n-2\}$ or, equivalently, for $k\in\{2,\ldots,n\}$.

\begin{remark}
It should also be noted that $\frac{\partial (f\circ \varphi_i)}{\partial e_k} = 2\,\frac{\partial(g\circ \varphi_i)}{\partial e_{k+1}}$.
\end{remark}

As with the SSLI, we can now infer the monotonicity of $g \circ \varphi$ along straight lines, which yields the following result.
\begin{corollary}\label{isotropieungleichung}
Let $n\in\N$ and $x_1,x_2,\ldots,x_n$, $y_1,y_2,\ldots,y_n > 0$ such that\\[-1.5em]
\begin{align*}
\hspace{11.0em}e_1(x)\ &\boldsymbol{=}\ e_1(y)\\[-1em]
\intertext{and}
~\\[-3.5em]
e_k(x)\ &\leq\ e_k(y)\qquad\textrm{for all\ \,$k\in\{2,\ldots,n\}$}\,.\\[.0em]
\intertext{Then}
~\\[-4.55em]
-\sum_{i=1}^n x_i\,\log x_i\ &\leq\ -\sum_{i=1}^n y_i\,\log y_i\notag\,.
\end{align*}
\end{corollary}
Weakening the equality in the first condition to $e_k(x)\leq e_k(y)$ for all $k\in\{1,\ldots,n\}$ again allows for obvious counterexamples similar to the SSLI: For $x=(e^{-1},\ldots e^{-1})\in\R_+^n$ and $y=(1,\ldots,1)\in\R_+^n$ the weakened condition is true, yet $-\sum_{i=1}^nx_i\,\log x_i = \frac ne > 0 = -\sum_{i=1}^ny_i\,\log y_i$.

See also the discussion in Section~\ref{sec:appl}.

%
%
\subsection{The SSLI in terms of matrix invariants}
\label{sec:ssli-matrix}
Let $U\in\PSym(n)$, where $\PSym(n)\subset \R^{n\times n}$ denotes the set of positive definite symmetric $n\times n$-matrices. Then $U$ is orthogonally diagonalizable with real eigenvalues $\lambda_1,\dotsc,\lambda_n>0$. The $k$-th invariant $I_k(U)$ of $U$ is defined as the $k$-th elementary symmetric polynomial of the vector $\lambda(U)=(\lambda_1,\dotsc,\lambda_n)$, i.e.\ $I_k(U) \colonequals e_k\bigl(\lambda(U)\bigr)$; (thus, $I_1(U)=\tr X$ and $I_n(U)=\det U$).

The SSLI can be equivalently expressed in terms of these invariants of positive definite symmetric matrices.
\begin{theorem}
\label{theorem:SSLImatrixFormulation}
Let $U,\widetilde U\in\PSym(n)$. If $I_k(U) \leq I_k(\widetilde U)$ for all $k\in\{1,\ldots,n-1\}$ and $\det U \boldsymbol{=} \det \widetilde U$, then $\|\log U\|^2 \leq \|\log \widetilde U\|^2$, where $\log$ is the principal matrix logarithm on $\PSym(n)$ and $\norm{\,.\,}$ denotes the Frobenius matrix norm.

\end{theorem}
\begin{proof}
Since $\norm{\log U}^2=\sum_{i=1}^n (\log \lambda_i(U))^2$, using the SSLI we immediately have
\begin{equation*}
\|\log U\|^2\ = \sum_{i=1}^n\bigl(\log \lambda_i(U)\bigr)^2\ \leq\ \sum_{i=1}^n\bigl(\log \lambda_i(\widetilde U)\bigr)^2 = \|\log \widetilde U\|^2\,.\qedhere
\end{equation*}
\end{proof}

Theorem \ref{theorem:SSLImatrixFormulation} can be applied directly to the \emph{quadratic Hencky energy}
\begin{equation*}
\WH(F)\ =\ \mu \,\|\dev_n \log U\|^2 + \frac{\kappa}{2}\,[\tr (\log U)]^2\ =\ \mu \,\|\log U\|^2+ \frac{\lambda}{2}\,[\log(\det U)]^2\,,
\end{equation*}
which was introduced into the theory of nonlinear elasticity in 1929 by H.\ Hencky \cite{hencky1929}. Here, $F\in\GL^+(n)$ is the deformation gradient, $\GL^+(n)$ is the set of invertible $n\times n$-matrices with positive determinant, $U=\sqrt{F^TF}$ is the right stretch tensor and $\dev_n\log {U} =\log {U}-\frac{1}{n}\,\tr(\log {U})\cdot\id$ is the deviatoric part of the Hencky strain tensor $\log U$. The material parameters $\mu,\lambda$ with $\mu>0$ and $3\,\lambda+2\,\mu\geq0$ are called the Lam\'e constants, while $\kappa\geq0$ is known as the bulk modulus. The Hencky energy has recently been characterized by a unique geometric property \cite{neff_eidel_martin_2015,Neff_Eidel_Osterbrink_2013,NeffGhibaLankeit}: it measures the squared geodesic distance of $F$ to the special orthogonal group $\SO(n)$ with respect to a left-$\GLn$-invariant, right-$\SOn$-invariant Riemannian metric on $\GLn$.

In terms of the quadratic Hencky energy, Theorem \ref{theorem:SSLImatrixFormulation} can be stated as follows:
\begin{corollary}\label{hencky}
Let $F,\widetilde F\in\GL^+(n)$ with $U=\sqrt{F^TF}$ and $\widetilde U = \sqrt{\widetilde F^T\widetilde F}$. If $\det U = \det \widetilde U$ and $I_k(U)\leq I_k(\widetilde U)$ for all $k\in\{1,\dotsc,n-1\}$, then $\WH(F)\leq \WH(\widetilde F)$.
\end{corollary}
\begin{remark}
Corollary \ref{hencky} means that $W_H$ satisfies a version of Truesdell's empirical inequalities \cite[pages 158, 171]{truesdell2004non}.
\end{remark}

In a similar way, Corollary \ref{isotropieungleichung} can be stated in terms of matrix invariants.
\begin{theorem}\label{matrixinvariants}
Let $U,\widetilde U\in\PSym(n)$. If $\tr U \boldsymbol{=} \tr \widetilde U$ and $I_k(U) \leq I_k(\widetilde U)$ for all $k\in\{2,\ldots,n\}$, then $\innerproduct{U,\log U - \id} \geq \innerproduct{\smash{\widetilde U},\log \smash{\widetilde U} - \id}$, where $\innerproduct{X,Y}=\tr(X^TY)$ denotes the canonical inner product of two $n\times n$-matrices $X$ and $Y$.
\end{theorem}
\begin{proof}
Corollary \ref{isotropieungleichung} implies $\sum_{i=1}^n\lambda_i(U)\,\log \lambda_i(U)\ \geq\ \sum_{i=1}^n\log \lambda_i(\widetilde U)\,\log \lambda_i(\widetilde U)$. Using the condition $\tr U = \tr \widetilde U$, we compute
\begin{align*}
\innerproduct{U,\log U - \id}\ = \innerproduct{U,\log U} - \innerproduct{U,\id} &= \sum_{i=1}^n\lambda_i(U) \cdot \log \lambda_i(U) - \tr U\\
&\geq \sum_{i=1}^n\lambda_i(\widetilde U) \cdot \log \lambda_i(\widetilde U) - \tr \widetilde U = \innerproduct{\smash{\widetilde U},\log \smash{\widetilde U} - \id}\,.\qedhere
\end{align*}
\end{proof}

\begin{remark}
The constitutive law induced by the hyperelastic energy potential
\begin{equation*}
W_{\mathrm{B}}(F)\ =\ \innerproduct{U,\log U - \id}\ =\ \innerproduct{\sqrt{F^TF},\log \sqrt{F^TF} - \id}
\end{equation*}
is a special case of \emph{Becker's law} of elasticity, which was introduced by the geologist G.F.\ Becker in 1893 \cite{becker1893} in a way remarkably similar to H.\ Hencky's deduction of the quadratic Hencky energy \cite{neff2014becker}. Becker's elastic law is hyperelastic (i.e.\ admits an energy potential) only in the lateral contraction free case, which is described by the energy function $W_{\mathrm{B}}$.
\end{remark}

Since $\tr (X\log X) = \sum_{i=1}^n \lambda_i(X)\log \lambda_i(X)$, we can translate the last Theorem \ref{matrixinvariants} into a statement for the quantum von Neumann entropy.
\begin{corollary}
Let $X, Y\in\PSym(n)$ be density matrices, so that $\tr X \boldsymbol{=} \tr Y = 1$. If $I_k(X) \leq I_k(Y)$ for all $k\in\{2,\ldots,n\}$, then the von Neumann entropy $-\tr (X\log X) \le -\tr(Y \log Y)$.
\end{corollary}

%
%
\subsection{Application to geodesic distance on $\PSym(n)$}
\label{sec:geod}
The convex cone $\PSym(n)$ is frequently also viewed as a \emph{Riemannian manifold} endowed with the Riemannian metric \cite{bhatia2006riemannian,fiala2011geometrical,moakher2005differential,moakher2006closest,rougee2006intrinsic}
\begin{equation}
g_C(X,Y) = \tr(C^{-1} X C^{-1} Y), 
\end{equation}
where $C\in\PSym(n)$ and $X,Y\in\Sym(n)=T_C\PSym(n)$, the tangent space at $C$. This manifold is geodesically complete, and the unique geodesic joining $C_1,C_2 \in \PSym(n)$ has the closed form~\citep[6.1.6]{bhatia2009positive}
\begin{equation}
\label{eq:PSymGeodesicCurves}
\gamma(t) \;=\; C_1^{\nicefrac 12}(C_1^{-\nicefrac 12} C_2 C_1^{-\nicefrac 12})^{t}\,C_1^{\nicefrac 12},\qquad t \in [0,1].
\end{equation}
The \emph{geodesic distance} between $C_1,C_2\in\PSym(n)$ is given by~\citep[6.1.6]{bhatia2009positive}
\begin{equation*}
\dgPsym (C_1,C_2) = \norm{\log(C_2^{-\nicefrac12} \,C_1\, C_2^{-\nicefrac12})}\,.
\end{equation*}
In particular, for $C_2=\id$, we obtain the simple formula
\begin{equation*}
\dgPsym^2(C_1,\id) = \norm{\log C_1}^2\,.
\end{equation*}
The sum-of-squared-logarithm inequalities can therefore be stated in terms of the geodesic distance of positive definite symmetric matrices to the identity matrix $\id$.
\begin{corollary}
\label{cor:geodesicDistance}
Let $I_1(C),\dotsc,I_n(C),\,I_1(\widetilde C),\dotsc,I_n(\widetilde C)$ denote the principal invariants of $C,\widetilde C\in\PSym(n)$. If $I_k(C) \leq I_k(\widetilde C)$ for all $k\in\{1,\dotsc,n-1\}$ and $I_n(C) = \det C = \det \widetilde C = I_n(\widetilde C)$, then
\begin{equation*}
\dgPsym(C,\id) \leq \dgPsym(\widetilde C,\id)\,.
\end{equation*}
\end{corollary}
Since $\id$ commutes with every matrix, $\dgPsym$ actually reduces to the \emph{log-Euclidean distance}
\begin{equation*}
\distlogeuc(C,\widetilde C) \colonequals \norm{\log C - \log \widetilde C}\,,
\end{equation*}
i.e., the Euclidean distance between the principal logarithms of two matrices \citep{arsigny2007geometric}. Thus, the SSLI may be equivalently stated for $\distlogeuc$.

%
%
\subsection{Additional applications}
The SSLI may also find applications in other fields. In the following we list some of these potential applications.
\begin{itemize}
\setlength{\itemsep}{1pt}

\item Recall that $I_k(X)=\tr \wedge^k X$, where $\wedge$ denotes the antisymmetric tensor product (Grassmann product). Thus, the SSLI for matrices can be equivalently stated in terms of this tensor product. This notation immediately suggests an interesting generalization, namely to the dual case of \emph{symmetric tensor product} denoted $\tr \vee^k X=h_k(\lambda(X))$, where $h_k$ denotes the $k$-th complete symmetric polynomial. Thus, we can consider inequalities of the form
\begin{align*}
&\tr \vee^k X \leq \tr \vee^k Y\ \ \textrm{for all $k\in\{1,\ldots,n-1\}$}\quad \textrm{and}\quad \tr \vee^n X = \tr \vee^n Y\\[.5em]
&\qquad\qquad\qquad\qquad \implies\qquad F(X) \leq F(Y)\,.
\end{align*}
More generally, extensions of the SSLI to other matrix monotone functions may be obtained by building on~\citep{silhavy_2015}.
\item It might also be of interest to characterize the matrix functions $\mathcal F\colon \R^{n\times n}\to \R^{n\times n}$ for which
\begin{equation*}
\inf_{Q\in\SO(n)} \| \sym \mathcal F(Q^TZ) \| = \inf_{Q\in\SO(n)} \| \mathcal F(Q^TZ) \| = \| \mathcal F(R^TZ) \| = \|\mathcal F(H)\|\,,
\end{equation*}
where $Z = R\,H$ is the polar decomposition of Z, see \cite{lankeit2014minimization,Neff_Nagatsukasa_logpolar13} and compare to \eqref{question}.
\item The final connection that we mention is perhaps the most interesting. Nonnegative elementary symmetric polynomials of matrices have been studied under the guise of $Q$-matrices ($Q_0$-matrices)~\citep{hershkowitz1983spectra}. These are complex matrices, whose elementary symmetric polynomials are positive (nonnegative). These matrices are much more tractable than the better known class of $P$-matrices (i.e.\ matrices with positive principal minors), which have been extensively studied in matrix analysis and optimization~\citep{fiedler1962matrices,berman1979nonnegative,coxson1994p}. The following theorem of Kellog applies to $P$ matrices, and also to $Q$ matrices~\citep{hershkowitz1983spectra}:
\begin{theorem}[\citep{kellogg1972complex}]
Let $X \in \C^{n\times n}$ be a matrix for which $I_k(X) \ge 0$ for $k\in\{1,\ldots,n\}$. Then, the spectrum of $A$ is contained in the set
\begin{equation*}
D \colonequals \left\lbrace z : |\arg z| \le \pi - \frac\pi n\right\rbrace.
\end{equation*}
Moreover, if any eigenvalue of $X$ lies on the boundary of $D$, then necessarily all symmetric functions $I_k(X)$, except $I_n(X)$, are equal to zero.
\end{theorem}  
\end{itemize}

%
%
\section{Alternative proof of Conjecture \ref{parpos}}
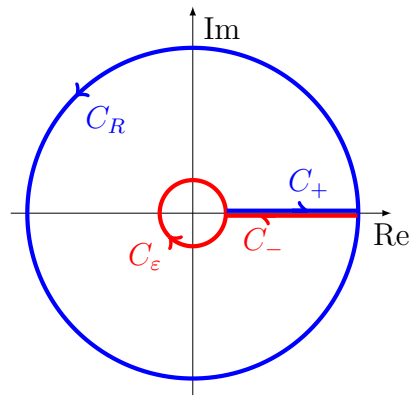
\begin{wrapfigure}[14]{r}{0.41\textwidth}\vspace{-1cm}
\begin{center}
    \begin{tikzpicture}
        \begin{scope}[scale=2.2]
\tikzstyle arrowstyle=[scale=27]
\draw[color=black, very thin, -latex] (-1.1,0) -- (1.2,0) node[below] {$\textrm{Re}$};
\draw[color=black, very thin, -latex] (0,-1.1) -- (0,1.25) node[below right] {$\textrm{Im}$};

\draw[color=red,ultra thick] plot[domain=0:225,samples=100] ({.2*cos(\x)},{.2*sin(\x)}) node[below left = -2] {$C_\varepsilon$};;
\draw[color=red,ultra thick,to-] plot[domain=224:360,samples=100] ({.2*cos(\x)},{.2*sin(\x)});
\draw[color=blue,ultra thick,-to] plot[domain=0:135,samples=100] ({1*cos(\x)},{1*sin(\x)}) node[below right = -1] {$C_R$};
\draw[color=blue,ultra thick] plot[domain=134:360,samples=100] ({1*cos(\x)},{1*sin(\x)});
\draw[color=blue,ultra thick,-left to] (.205,.015) -- (.7,.015) node[above] {$C_+$};;
\draw[color=blue,ultra thick] (.6,.015) -- (1,.015);
\draw[color=red,ultra thick,-left to] (.995,-.015) -- (.4,-.015) node[below] {$\;C_-$};;
\draw[color=red,ultra thick] (.5,-.015) -- (.2,-.015);

\end{scope}
    \end{tikzpicture}
    \caption{\label{fig:complex_plane}The path $C$ consists of the following four pieces $C_\eps$, $C_+$, $C_R$ and $C_-$.}
\end{center}
\end{wrapfigure}

In Section \ref{section:proof}, we have only shown Conjecture \ref{parpos} under the additional assumption that $e\in\R_+^n$ corresponds to pairwise different roots $\varphi(e)$ which was sufficient to prove the SSLI. Nevertheless, it is striking that the expression of the partial derivatives

\begin{equation}
\frac{\partial(f\circ\varphi)}{\partial e_k}(e) = 2\int_0^\infty \frac{t^{n-k-1}}{\prod_{j=1}^n(t+\varphi_j(e))}\intd t
\end{equation}

can be extended continuously to $e\in\R_+^n$ with multiple roots. This strongly suggests that Conjecture \ref{parpos} indeed holds for all $e\in\R_+^n$. We now restate the conjecture as a lemma and give a short proof for the general case.

\begin{lemma}
\label{borisov}
For $a=(a_1,\ldots,a_{n-1})\in \R_+^{n-1}$ let $\widehat{h}_a(z) = z^n + a_{n-1}\,z^{n-1} + \ldots + a_1\,z + 1$. Define the function $\widehat f\colon \R_+^{n-1}\to\R$ by
\begin{equation*}
\widehat f(a_1,\ldots,a_{n-1})\ =\ \sum_{\substack{\widehat z\in\C\\\widehat{h}_a(\widehat{z})=0}}\bigl(\log(-\widehat{z})\bigr)^2\,,
\end{equation*}

where $\log$ denotes the principal branch of the complex logarithm on $\C\setminus\R_{\leq 0}$ and roots are counted with multiplicities.  Then all partial derivatives $\frac{\partial \widehat f}{\partial a_k}(a)$ for any $a\in\R_+^{n-1}$ are positive.
\end{lemma}

Note carefully that the roots $z_1,\dotsc,z_n$ of $h_a$ in the original formulation of  Conjecture \ref{parpos} correspond to the negative $-\widehat{z}_1,\dotsc,-\widehat{z}_n$ of the roots of $\widehat{h}_a$ in Lemma \ref{borisov}.

The function $\widehat f$ is well defined, since none of the roots of $\widehat{h}_a$ are positive real numbers. Furthermore, $\widehat f$ is real-valued, because the roots $\widehat{z}$ of $\widehat{h}_a$ only occur as real numbers or complex conjugate pairs, which means that the values of $\bigl(\log(-\widehat{z})\bigr)^2$ are also real or complex conjugate pairs.

\begin{proof}
We write the function $\widehat f$ as a contour integral. In order to accommodate the multiple values of $\log(-z)$ we consider the Riemann surface with boundary obtained  by gluing two copies of $\R_+$ to $\C\setminus \R_{\geq 0}$, where $R_{\geq 0} \colonequals \{x\in\R\,|\, x\geq 0\}$. We will view the two copies as \emph{north} and \emph{south} copies $\R_+$ depending on whether they are approached from above and from below respectively. Note that the $\log(-z)$ function extends to the following 

\begin{equation*}
\log(-r_{\mathrm{north}})=\log r -\pi i,~~~\log(-r_{\mathrm{south}}) = \log r + \pi i 
\end{equation*}
on the north and south copies of $\R_+$.

In a neighborhood of fixed $a\in(\R_{\geq 0})^{n-1}$, for any sufficiently large $R$ and small enough $\eps>0$, we can apply the generalized argument principle \cite{ammari2009layer} to find\footnote{A related representation to \eqref{fa} has been used recently in \cite{jozsa2015symmetric} in connection with the entropy function.}
\begin{equation}\label{fa}
\widehat f(a)\ =\ \frac 1{2\pi i}\int_C\bigl(\log(-z)\bigr)^2 \, \frac{\widehat{h}_a'(z)}{\widehat{h}_a(z)} \, \intd z\,,
\end{equation}

where the path $C$ consists of the following four pieces $C_\eps$, $C_+$, $C_R$ and $C_-$, see also Figure \ref{fig:complex_plane}:

\begin{itemize}\setlength{\itemsep}{0.4em}
\item $C_\eps$ is the circle of radius $\eps$ around $0$ travelled  clockwise from $\eps_{\mathrm{south}}$ to $\eps_{\mathrm{north}}$;
\item $C_+$ is the line segment from $\eps_{\mathrm{north}}$ to $R_{\mathrm{north}}$ on the north copy of the  positive real axis;
\item $C_R$ is the circle of radius $R$ around $0$ travelled counter-clockwise from $R_{\mathrm{north}}$ to $R_{\mathrm{south}}$;
\item $C_-$ is the line segment from $R_{\mathrm{south}}$ to $\eps_{\mathrm{south}}$ on the south copy of the positive axis.
\end{itemize}
By straightforward computation, we find
\begin{align}
\frac{\partial}{\partial a_k}\,\frac{\widehat h'_a(z)}{\widehat h_a(z)}\ &=\ \frac{\frac{\partial}{\partial a_k}\,\widehat h'_a(z)\cdot \widehat h_a(z) - \widehat h'_a(z)\cdot \frac{\partial}{\partial a_k}\,\widehat h'_a(z)}{\bigl(\widehat h_a(z)\bigr)^2}\notag\\
&=\ \frac{k\,z^k\cdot \widehat h_a(z) - z^k\cdot \widehat h'_a(z)}{\bigl(\widehat h_a(z)\bigr)^2}\ \,=\ \left(\frac{z^k}{\widehat h_a(z)}\right)'\,.\label{strcomp}
\end{align}

The integrand is analytic on the compact path $C$, hence we can differentiate inside the integral:
\begin{align}
\frac{\partial \widehat f(a)}{\partial a_k}\ &=\ \frac 1{2\pi i}\int_C\bigl(\log(-z)\bigr)^2 \; \frac{\partial}{\partial a_k}\,\frac{\widehat{h}_a'(z)}{\widehat{h}_a(z)}\,\intd z\ \overset{\eqref{strcomp}}=\ \frac 1{2\pi i}\int_C\bigl(\log(-z)\bigr)^2\left(\frac{z^k}{\widehat{h}_a(z)}\right)'\intd  z\notag\\
&=\ -\frac 1{2\pi i}\int_C\left(\bigl(\log(-z)\bigr)^2\right)'\frac{z^{k}}{\widehat{h}_a(z)}\,\intd z\ =\ -\frac 1{\pi i}\int_C\log(-z)\frac{z^{k-1}}{\widehat{h}_a(z)}\,\intd z .
\end{align}
We can now take the limits as $R\to +\infty$ and $\eps \to 0$. Since $k\leq n-1$, the integral over $C_R$ tends to zero (the length of $C_R$ is $2\pi R$, while the function is of order $O(R^{-2}\log R)$). The integral over $C_\eps$ also tends to zero, since for $k\geq 1$, the integrand is of order $O(\log \eps)$ and the length of $C_\eps$ is $2\pi\eps$.

We conclude:
\begin{align}
\frac{\partial f(a)}{\partial a_k}\ &=\ -\frac 1{\pi i}\,\lim_{\eps\searrow 0,R\to\infty}\left(\int_{C_+\cup C_-}\!\log(-z)\frac{z^{k-1}}{\widehat{h}_a(z)}\intd z\right)\notag\\
&=\ -\frac 1{\pi i}\,\lim_{\eps\searrow 0,R\to\infty}\left(\int_{\textrm{$[\eps,R]$}}\bigl((\log t-\pi i)-(\log t+\pi i)\bigr)\frac{t^{k-1}}{\widehat{h}_a(t)}\intd t\right)\\
&=\ 2\,\lim_{\eps\searrow 0,R\to\infty}\int_\eps^R\frac{t^{k-1}}{\widehat{h}_a(t)}\intd t=2\int_0^{+\infty}\frac{t^{k-1}}{\widehat{h}_a(t)}\intd t\ >\ 0\,.\qedhere
\end{align}
\end{proof}

\section*{Acknowledgement}
L.B. was partially supported by NSF grant DMS-1201466. We are grateful to Johannes Lankeit (Universität Paderborn) for fruitful discussions on the SSLI. We also want to thank Prof.\ Georg Hein (Universität Duisburg-Essen) for his suggestions on algebraic topics. Our special thanks go out to Robert Martin (Universität Duisburg-Essen), who is always a great support with his commitment and considerable technical competence.


{\footnotesize\begin{thebibliography}{32}
\providecommand{\natexlab}[1]{#1}
\providecommand{\url}[1]{\texttt{#1}}
\expandafter\ifx\csname urlstyle\endcsname\relax
  \providecommand{\doi}[1]{doi: #1}\else
  \providecommand{\doi}{doi: \begingroup \urlstyle{rm}\Url}\fi

\bibitem[Ammari et~al.(2009)Ammari, Kang, and Lee]{ammari2009layer}
H.~Ammari, H.~Kang, and H.~Lee.
\newblock \emph{Layer potential techniques in spectral analysis}, volume 153.
\newblock American Mathematical Society Providence, 2009.
\newblock chapter 1 available at
  \url{www.ams.org/bookstore/pspdf/surv-153-prev.pdf}.

\bibitem[Arsigny et~al.(2007)Arsigny, Fillard, Pennec, and
  Ayache]{arsigny2007geometric}
V.~Arsigny, P.~Fillard, X.~Pennec, and N.~Ayache.
\newblock Geometric means in a novel vector space structure on symmetric
  positive-definite matrices.
\newblock \emph{SIAM Journal on Matrix Analysis and Applications}, 29\penalty0
  (1):\penalty0 328--347, 2007.

\bibitem[Becker(1893)]{becker1893}
G.~Becker.
\newblock The {F}inite {E}lastic {S}tress-{S}train {F}unction.
\newblock \emph{American Journal of Science}, 46:\penalty0 337--356, 1893.
\newblock newly typeset version available at
  \url{www.uni-due.de/imperia/md/content/mathematik/ag_neff/becker_latex_new1893.pdf}.

\bibitem[Berman and Plemmons(1979)]{berman1979nonnegative}
A.~Berman and R.~J. Plemmons.
\newblock Nonnegative matrices.
\newblock \emph{The Mathematical Sciences, Classics in Applied Mathematics}, 9,
  1979.

\bibitem[Bhatia(2009)]{bhatia2009positive}
R.~Bhatia.
\newblock \emph{Positive definite matrices}.
\newblock Princeton University Press, 2009.

\bibitem[Bhatia and Holbrook(2006)]{bhatia2006riemannian}
R.~Bhatia and J.~Holbrook.
\newblock Riemannian geometry and matrix geometric means.
\newblock \emph{Linear Algebra and its Applications}, 413\penalty0
  (2):\penalty0 594--618, 2006.

\bibitem[B{\^{\i}}rsan et~al.(2013)B{\^{\i}}rsan, Neff, and
  Lankeit]{Neff_log_inequality13}
M.~B{\^{\i}}rsan, P.~Neff, and J.~Lankeit.
\newblock Sum of squared logarithms -- an inequality relating positive definite
  matrices and their matrix logarithm.
\newblock \emph{Journal of Inequalities and Applications}, 2013\penalty0
  (168):\penalty0 1--16, 2013.
\newblock open access.

\bibitem[Coxson(1994)]{coxson1994p}
G.~E. Coxson.
\newblock {The P-matrix problem is co-NP-complete}.
\newblock \emph{Mathematical Programming}, 64\penalty0 (1-3):\penalty0
  173--178, 1994.

\bibitem[Cucker and Corbalan(1989)]{cucker1989alternate}
F.~Cucker and A.~G. Corbalan.
\newblock An alternate proof of the continuity of the roots of a polynomial.
\newblock \emph{American Mathematical Monthly}, pages 342--345, 1989.

\bibitem[Dannan et~al.(2015)Dannan, Neff, and Thiel]{DannanNeffThiel2015}
F.~Dannan, P.~Neff, and C.~Thiel.
\newblock On the sum of squared logarithms inequality and related inequalities.
\newblock \emph{Journal of Mathematical Inequalities}, 2015.
\newblock open access, available at arXiv:1411.1290.

\bibitem[Fiala(2011)]{fiala2011geometrical}
Z.~Fiala.
\newblock Geometrical setting of solid mechanics.
\newblock \emph{Annals of Physics}, 326\penalty0 (8):\penalty0 1983--1997,
  2011.

\bibitem[Fiedler and Ptak(1962)]{fiedler1962matrices}
M.~Fiedler and V.~Ptak.
\newblock On matrices with non-positive off-diagonal elements and positive
  principal minors.
\newblock \emph{Czechoslovak Mathematical Journal}, 12\penalty0 (3):\penalty0
  382--400, 1962.

\bibitem[Hencky(1929)]{hencky1929}
H.~Hencky.
\newblock {W}elche {U}mst{\"a}nde bedingen die {V}erfestigung bei der bildsamen
  {V}erformung von festen isotropen {K\"o}rpern?
\newblock \emph{Zeitschrift f{\"u}r {P}hysik}, 55:\penalty0 145--155, 1929.
\newblock available at
  \url{www.uni-due.de/imperia/md/content/mathematik/ag_neff/hencky1929.pdf}.

\bibitem[Hershkowitz(1983)]{hershkowitz1983spectra}
D.~Hershkowitz.
\newblock On the spectra of matrices having nonnegative sums of principal
  minors.
\newblock \emph{Linear Algebra and its Applications}, 55:\penalty0 81--86,
  1983.

\bibitem[Jozsa and Mitchison(2015)]{jozsa2015symmetric}
R.~Jozsa and G.~Mitchison.
\newblock Symmetric polynomials in information theory: Entropy and subentropy.
\newblock \emph{Journal of Mathematical Physics}, 56\penalty0 (6):\penalty0
  062201, 2015.

\bibitem[Kellogg(1972)]{kellogg1972complex}
R.~Kellogg.
\newblock {On complex eigenvalues of $M$ and $P$ matrices}.
\newblock \emph{Numerische Mathematik}, 19\penalty0 (2):\penalty0 170--175,
  1972.

\bibitem[Lang(2002)]{lang2002algebra}
S.~Lang.
\newblock \emph{Algebra}.
\newblock Springer, revised third edition, 2002.

\bibitem[Lankeit et~al.(2014)Lankeit, Neff, and
  Nakatsukasa]{lankeit2014minimization}
J.~Lankeit, P.~Neff, and Y.~Nakatsukasa.
\newblock The minimization of matrix logarithms: On a fundamental property of
  the unitary polar factor.
\newblock \emph{Linear Algebra and its Applications}, 449:\penalty0 28--42,
  2014.

\bibitem[Mitchison and Jozsa(2004)]{mitchison2004towards}
G.~Mitchison and R.~Jozsa.
\newblock Towards a geometrical interpretation of quantum-information
  compression.
\newblock \emph{Physical Review A}, 69\penalty0 (3):\penalty0 032304, 2004.

\bibitem[Moakher(2005)]{moakher2005differential}
M.~Moakher.
\newblock A differential geometric approach to the geometric mean of symmetric
  positive-definite matrices.
\newblock \emph{SIAM Journal on Matrix Analysis and Applications}, 26\penalty0
  (3):\penalty0 735--747, 2005.

\bibitem[Moakher and Norris(2006)]{moakher2006closest}
M.~Moakher and A.~N. Norris.
\newblock The closest elastic tensor of arbitrary symmetry to an elasticity
  tensor of lower symmetry.
\newblock \emph{Journal of Elasticity}, 85\penalty0 (3):\penalty0 215--263,
  2006.

\bibitem[Neff(2015)]{sslimathoverflow}
P.~Neff.
\newblock The sum of squared logarithms conjecture, May 2015.
\newblock
  \url{www.mathoverflow.net/questions/207845/the-sum-of-squared-logarithms-conjecture}.

\bibitem[Neff et~al.(2013)Neff, Eidel, Osterbrink, and
  Martin]{neff2013henckyPAMM}
P.~Neff, B.~Eidel, F.~Osterbrink, and R.~Martin.
\newblock The {H}encky strain energy {$\|\log U\|^2$} measures the geodesic
  distance of the deformation gradient to {$\SO(n)$} in the canonical
  left-invariant {R}iemannian metric on {$\GL(n)$}.
\newblock \emph{Proceedings in Applied Mathematics and Mechanics}, 13\penalty0
  (1):\penalty0 369--370, 2013.

\bibitem[Neff et~al.(2014{\natexlab{a}})Neff, Eidel, Osterbrink, and
  Martin]{Neff_Eidel_Osterbrink_2013}
P.~Neff, B.~Eidel, F.~Osterbrink, and R.~J. Martin.
\newblock A {R}iemannian approach to strain measures in nonlinear elasticity.
\newblock \emph{Comptes Rendus M{\'e}canique}, 342\penalty0 (4):\penalty0
  254--257, 2014{\natexlab{a}}.

\bibitem[Neff et~al.(2014{\natexlab{b}})Neff, M{\"u}nch, and
  Martin]{neff2014becker}
P.~Neff, I.~M{\"u}nch, and R.~Martin.
\newblock Rediscovering {G.F.} {B}ecker's early axiomatic deduction of a
  multiaxial nonlinear stress-strain relation based on logarithmic strain.
\newblock \emph{to appear in Mathematics and Mechanics of Solids, doi:
  10.1177/1081286514542296}, 2014{\natexlab{b}}.
\newblock \doi{10.1177/1081286514542296}.
\newblock available at arXiv:1403.4675.

\bibitem[Neff et~al.(2014{\natexlab{c}})Neff, Nakatsukasa, and
  Fischle]{Neff_Nagatsukasa_logpolar13}
P.~Neff, Y.~Nakatsukasa, and A.~Fischle.
\newblock A logarithmic minimization property of the unitary polar factor in
  the spectral norm and the {F}robenius matrix norm.
\newblock \emph{SIAM Journal on Matrix Analysis and Applications}, 35\penalty0
  (3):\penalty0 1132--1154, 2014{\natexlab{c}}.

\bibitem[Neff et~al.(2015{\natexlab{a}})Neff, Eidel, and
  Martin]{neff_eidel_martin_2015}
P.~Neff, B.~Eidel, and R.~J. Martin.
\newblock Geometry of logarithmic strain measures in solid mechanics.
\newblock \emph{arXiv preprint arXiv:1505.02203, submitted},
  2015{\natexlab{a}}.

\bibitem[Neff et~al.(2015{\natexlab{b}})Neff, Ghiba, and
  Lankeit]{NeffGhibaLankeit}
P.~Neff, I.~D. Ghiba, and J.~Lankeit.
\newblock The exponentiated {H}encky-logarithmic strain energy. {P}art {I}:
  {C}onstitutive issues and rank-one convexity.
\newblock \emph{to appear in Journal of Elasticity}, 2015{\natexlab{b}}.
\newblock available at arXiv:1403.3843.

\bibitem[Pompe and Neff(2015)]{PompeNeff2015}
W.~Pompe and P.~Neff.
\newblock On the generalized sum of squared logarithms inequality.
\newblock \emph{Journal of Inequalities and Applications}, 2015.
\newblock open access, available at arXiv:1410.2706.

\bibitem[Roug{\'e}e(2006)]{rougee2006intrinsic}
P.~Roug{\'e}e.
\newblock An intrinsic {L}agrangian statement of constitutive laws in large
  strain.
\newblock \emph{Computers and Structures}, 84\penalty0 (17):\penalty0
  1125--1133, 2006.

\bibitem[{\v{S}}ilhav{\'y}(2015)]{silhavy_2015}
M.~{\v{S}}ilhav{\'y}.
\newblock A functional inequality related to analytic continuation.
\newblock \emph{preprint Institute of Mathematics AS CR IM-2015-37}, 2015.
\newblock available at
  \url{www.math.cas.cz/fichier/preprints/IM_20150623102729_44.pdf}.

\bibitem[Truesdell and Noll(2004)]{truesdell2004non}
C.~Truesdell and W.~Noll.
\newblock \emph{The non-linear field theories of mechanics}.
\newblock Springer, 2004.
\newblock originally published as Volume III/3 of the Encyclopedia of Physics
  in 1965.

\end{thebibliography}
}

\end{document}